\documentclass[leqno,12pt]{amsart} 
\usepackage{amssymb} 
\usepackage[all]{xy}
\usepackage{graphicx}
\theoremstyle{plane} 
\newtheorem{theorem}{\indent\sc Theorem}[section] 
\newtheorem{lemma}[theorem]{\indent\sc Lemma}

\newtheorem{proposition}[theorem]{\indent\sc Proposition}
\newtheorem{claim}[theorem]{\indent\sc Claim}
\newtheorem{problem}[theorem]{\indent\sc Problem}
\newtheorem*{main1}{\indent\sc Main Theorem 1}
\newtheorem*{main2}{\indent\sc Main Theorem 2}

\theoremstyle{definition}

\newtheorem{example}[theorem]{\indent\sc Example}

\begin{document}

\title[On exterior power of tangent bundles of Fano manifolds]{On the second and third exterior power of tangent bundles of Fano manifolds with birational contractions}
\author[K. Yasutake]{Kazunori Yasutake} 
\date{\today}
\address{Organization for the Strategic Coordination of Research and 
Intellectual Properties, Meiji University, Kanagawa 214-8571 Japan}
\email{tz13008@meiji.ac.jp}
\maketitle

\begin{abstract}
In this paper, we classify Fano manifolds with elementary contractions of birational type such that the second or third exterior power of tangent bundles are numerically effective.
\end{abstract}

\section*{Introduction}
In the paper \cite{peternell}, F. Campana and T. Peternell classified smooth projective threefolds $X$ such that the second exterior power of tangent bundle $\wedge^2\mathcal{T}_X$ are numerically effective $($nef, for short$)$. 
Among such threefolds, the blowing-up of three-dimensional projective space at a point is the only manifold having an elementary contraction of birational type.
In \cite{yasutake}, the author classified Fano fourfolds with the same condition and see that the blowing-up of four-dimensional projective space at a point is the only manifold having an elementary contraction of birational type.     
In this paper, we show that the same statement also holds in arbitrary dimension.  

\begin{main1}\label{second}
Let $X$ be an $n$-dimensional Fano manifold with at least one elementary contraction of birational type such that the second exterior power of tangent bundle $\wedge^2\mathcal{T}_X$ is nef where $n\geqslant 3$.
Then $X$ is isomorphic to the blowing-up of the $n$-dimensional projective space $\mathbb{P}^n$ at a point. 
\end{main1}  

Furthermore, we also study the case where the third exterior power of tangent bundle $\wedge^3\mathcal{T}_X$ is nef and obtain the following theorem.       
\begin{main2}\label{third}
Let $X$ be an $n$-dimensional Fano manifold with at least one elementary contraction of birational type such that the third exterior power of tangent bundle $\wedge^3\mathcal{T}_X$ is nef where $n\geqslant 4$.
Then X is isomorphic to one of the following,
\begin{enumerate}
\item blowing-up of $\mathbb{P}^n$ at a point, 
\item $\mathbb{P}_{\mathbb{P}^{n-1}}(\mathcal{O}_{\mathbb{P}^{n-1}}\oplus \mathcal{O}_{\mathbb{P}^{n-1}}(-2))$,
\item $\mathbb{P}_{Q^{n-1}}(\mathcal{O}_{Q^{n-1}}\oplus \mathcal{O}_{Q^{n-1}}(-1))$ where $Q^{n-1}$ is a $(n-1)$-dimensional smooth quadric hypersurface in $\mathbb{P}^n$, 
\item blowing-up of $\mathbb{P}^n$ along a line, 
\item product of $\mathbb{P}^1$ and the blowing-up of $\mathbb{P}^{n-1}$ at a point, 
\item blowing-up of $Q^4$ along a line, 
\item blowing-up of $Q^4$ along a conic not on a plain contained in $Q^4$, 
\end{enumerate}
\end{main2}

There is a problem about the nefness of $\wedge^q\mathcal{T}_X$ posed by F. Campana and T. Peternell.
\begin{problem}[\cite{peternell}, Problem 6.4]
Let X be a Fano manifold. Assume that $\wedge^q\mathcal{T}_X$ is nef on every extremal rational curve.
Is then $\wedge^q\mathcal{T}_X$ already nef ?
\end{problem}

In the proof of above theorems, we can see that the following theorem also holds.

\begin{theorem}
Let X be an $n$-dimensional Fano manifold with at least one elementary contraction of birational type where $n\geqslant 4$.
Assume that $\wedge^2\mathcal{T}_X$ $(resp. \wedge^3\mathcal{T}_X)$ is nef on every extremal rational curves in X.
Then $\wedge^2\mathcal{T}_X$ $(resp. \wedge^3\mathcal{T}_X)$ is nef.
\end{theorem}  

\section*{Acknowledgements}
The author would like to express his gratitude to Professor Eiichi Sato for many useful discussions and much warm encouragement. 
The author also would like to thank to Doctor Kento Fujita for many useful discussions. 
\section*{Notation}
Throughout this paper we work over the complex number field $\mathbb{C}$.
For a projective manifold X and a vector bundle $\mathcal{E}$ on $X$, let $\mathbb{P}_X(\mathcal{E})$ be the projectivization of $\mathcal{E}$ in the sense of Grothendieck and 
$\xi_{\mathcal{E}}$ be the tautological divisor. We usually denote the projection by $\pi: \mathbb{P}_X(\mathcal{E})\rightarrow X$.
We denote the blowing-up of smooth projective manifold $M$ along a smooth subvariety $Z$ by $Bl_Z(M)$.  

\section{Examples}
In this section, we confirm that manifolds appeared in Main Theorem 1 $($resp. Main Theorem 2$)$ satisfy the condition that  $\displaystyle \wedge ^2 \mathcal{T}_X$ $($resp. $\displaystyle \wedge ^3 \mathcal{T}_X)$ are nef. 
\begin{example}\label{bl_pt}
Let $X=Bl_{pt}(\mathbb{P}^n)$ be the blowing-up of $\mathbb{P}^n$ at a point where $n\geqslant4$. 
Then, by Lemma 2.1 in \cite{yasutake}, we know that $\wedge^2\mathcal{T}_X$ is nef.
\end{example}

\begin{example}\label{third1}
Set $X=\mathbb{P}_{\mathbb{P}^{n-1}}(\mathcal{O}_{\mathbb{P}^{n-1}}(2)\oplus \mathcal{O}_{\mathbb{P}^{n-1}})\cong\mathbb{P}_{\mathbb{P}^{n-1}}(\mathcal{O}_{\mathbb{P}^{n-1}}\oplus \mathcal{O}_{\mathbb{P}^{n-1}}(-2))$ where $n\geqslant4$. .
Put $\xi$ the tautological divisor on $\mathbb{P}_{\mathbb{P}^{n-1}}(\mathcal{O}_{\mathbb{P}^{n-1}}(2)\oplus \mathcal{O}_{\mathbb{P}^{n-1}})$ and $\pi : \mathbb{P}_{\mathbb{P}^{n-1}}(\mathcal{O}_{\mathbb{P}^{n-1}}(2)\oplus\mathcal{O}_{\mathbb{P}^{n-1}})\rightarrow \mathbb{P}^{n-1}$ the natural projection.
We have the following exact sequence :
\[0\rightarrow \mathcal{T}_{\pi}\rightarrow \mathcal{T}_X\rightarrow \pi^*\mathcal{T}_{\mathbb{P}^{n-1}}\rightarrow 0. \]
From this exact sequence, we obtain the following exact secuence : 
\[0\rightarrow \mathcal{T}_{\pi}\otimes\pi^*(\wedge^2\mathcal{T}_{\mathbb{P}^{n-1}})\rightarrow \wedge^3\mathcal{T}_X\rightarrow \pi^*(\wedge^3\mathcal{T}_{\mathbb{P}^{n-1}})\rightarrow 0. \]
The subbundle $\mathcal{T}_{\pi}\otimes\pi^*(\wedge^2\mathcal{T}_{\mathbb{P}^{n-1}})\cong\mathcal{O}_X(2\xi)\otimes\pi^*(\wedge^2\mathcal{T}_{\mathbb{P}^{2}}(-2))$ is nef.
Therefore, we have $\wedge^3\mathcal{T}_X$ is nef since $\pi^*(\wedge^3\mathcal{T}_{\mathbb{P}^{n-1}})$ is nef. 
\end{example}

\begin{example}\label{third2}
Set $X=\mathbb{P}_{Q^{n-1}}(\mathcal{O}_{Q^{n-1}}(1)\oplus \mathcal{O}_{Q^{n-1}})\cong\mathbb{P}_{Q^{n-1}}(\mathcal{O}_{Q^{n-1}}\oplus \mathcal{O}_{Q^{n-1}}(-1))$ where $n\geqslant4$. .
Put $\xi$ the tautological divisor on $\mathbb{P}_{Q^{n-1}}(\mathcal{O}_{Q^{n-1}}(1)\oplus \mathcal{O}_{Q^{n-1}})$ and $\pi : \mathbb{P}_{Q^{n-1}}(\mathcal{O}_{Q^{n-1}}(1)\oplus\mathcal{O}_{Q^{n-1}})\rightarrow Q^{n-1}$ the natural projection.
We have the following exact sequence :
\[0\rightarrow \mathcal{T}_{\pi}\rightarrow \mathcal{T}_X\rightarrow \pi^*\mathcal{T}_{Q^{n-1}}\rightarrow 0. \]
From this exact sequence, we obtain the following exact secuence : 
\[0\rightarrow \mathcal{T}_{\pi}\otimes\pi^*(\wedge^2\mathcal{T}_{Q^{n-1}})\rightarrow \wedge^3\mathcal{T}_X\rightarrow \pi^*(\wedge^3\mathcal{T}_{Q^{n-1}})\rightarrow 0. \]
We can show that the subbundle $\mathcal{T}_{\pi}\otimes\pi^*(\wedge^2\mathcal{T}_{Q^{n-1}})\cong\mathcal{O}_X(2\xi)\otimes\pi^*(\wedge^2\mathcal{T}_{Q^{n-1}}(-1))$ is nef by the surjection from nef bundle $\wedge^3\mathcal{T}_{\mathbb{P}^{n}}(-3)|_{Q^{n-1}}\rightarrow\wedge^2\mathcal{T}_{Q^{n-1}}(-1)\rightarrow 0$.
Hence, we have $\wedge^3\mathcal{T}_X$ is nef since $\pi^*(\wedge^3\mathcal{T}_{Q^{n-1}})$ is nef. 
\end{example}

\begin{example}\label{third3}
Let $X=Bl_{l}(\mathbb{P}^n)$ be the blowing-up of $\mathbb{P}^n$ along a line $l$ where $n\geqslant4$. 
Then, It is known that $X$ has a $\mathbb{P}^2$-bundle structure $X\cong\mathbb{P}_{\mathbb{P}^{n-2}}(\mathcal{O}^2_{\mathbb{P}^{n-2}}\oplus\mathcal{O}_{\mathbb{P}^{n-2}}(1))$. 
Put $\xi$ the tautological divisor 
and 
$\pi : \mathbb{P}_{\mathbb{P}^{n-2}}(\mathcal{O}^2_{\mathbb{P}^{n-2}}\oplus\mathcal{O}_{\mathbb{P}^{n-2}}(1))\rightarrow \mathbb{P}^{n-2}$ the natural projection.
We have the following two exact sequences :
\[0\rightarrow \mathcal{T}_{\pi}\rightarrow \mathcal{T}_X\rightarrow \pi^*\mathcal{T}_{\mathbb{P}^{n-2}}\rightarrow 0 \]
where $\mathcal{T}_{\pi}$ is the relative tangent bundle and   
\[0\rightarrow \mathcal{O}_{X}\rightarrow \mathcal{O}_X(\xi)\otimes\pi^*(\mathcal{O}^2_{\mathbb{P}^{n-2}}\oplus\mathcal{O}_{\mathbb{P}^{n-2}}(-1))\rightarrow \mathcal{T}_{\pi}\rightarrow 0. \]

If $n=4$, we have the following exact sequence from the first exact sequence : 
\[0\rightarrow \wedge^2\mathcal{T}_{\pi}\otimes\pi^*\mathcal{T}_{\mathbb{P}^{2}}\rightarrow \wedge^3\mathcal{T}_X\rightarrow \mathcal{T}_{\pi}\otimes\pi^*(\wedge^2\mathcal{T}_{\mathbb{P}^{2}})\rightarrow 0. \]
The subbundle $\wedge^2\mathcal{T}_{\pi}\otimes\pi^*\mathcal{T}_{\mathbb{P}^{2}}\cong\mathcal{O}_X(3\xi)\otimes\pi^*\mathcal{T}_{\mathbb{P}^{2}}(-1)$ is nef.
The quotient bundle is also nef by the surjection from the ample vector bundle 
$\mathcal{O}_X(\xi)\otimes\pi^*(\mathcal{O}^2_{\mathbb{P}^{n-2}}(3)\oplus\mathcal{O}_{\mathbb{P}^{n-2}}(2))\rightarrow \mathcal{T}_{\pi}\otimes\pi^*(\wedge^2\mathcal{T}_{\mathbb{P}^{2}})\rightarrow 0$. 
Therefore, $\wedge^3\mathcal{T}_X$ is nef.  

If $n\geqslant 5$, we have the following two exact sequences : 
\[0\rightarrow\mathcal{F}\rightarrow \wedge^3\mathcal{T}_X\rightarrow \pi^*(\wedge^3\mathcal{T}_{\mathbb{P}^{n-2}})\rightarrow 0 \]
and 
\[0\rightarrow \wedge^2\mathcal{T}_{\pi}\otimes\pi^*\mathcal{T}_{\mathbb{P}^{n-2}}\rightarrow \mathcal{F}\rightarrow \mathcal{T}_{\pi}\otimes\pi^*(\wedge^2\mathcal{T}_{\mathbb{P}^{n-2}})\rightarrow 0. \]
We can easily show that bundles $\pi^*(\wedge^3\mathcal{T}_{\mathbb{P}^{n-2}})$ and $\wedge^2\mathcal{T}_{\pi}\otimes\pi^*\mathcal{T}_{\mathbb{P}^{n-2}}$ are nef. 
We also see that $\mathcal{T}_{\pi}\otimes\pi^*(\wedge^2\mathcal{T}_{\mathbb{P}^{n-2}})$ is nef by the surjection from the ample vector bundle 
$\mathcal{O}_X(\xi)\otimes\pi^*(\wedge^2\mathcal{T}_{\mathbb{P}^{n-2}}^{\oplus 2}\oplus\wedge^2\mathcal{T}_{\mathbb{P}^{n-2}}(-1))\rightarrow \mathcal{T}_{\pi}\otimes\pi^*(\wedge^2\mathcal{T}_{\mathbb{P}^{n-2}})\rightarrow 0$. 
Therefore, $\wedge^3\mathcal{T}_X$ is nef.  
\end{example}

\begin{example}\label{third4}
Let $X=\mathbb{P}^1\times Y$ be the product of $\mathbb{P}^1$ and the blowing-up of $\mathbb{P}^{n-1}$ at a point $Y=Bl_{pt}(\mathbb{P}^{n-1})$ where $n\geqslant 4$.
Set $p : X\rightarrow \mathbb{P}^1$ the first projection and $q : X\rightarrow Bl_{pt}(\mathbb{P}^{n-1})$ the second projection.
Then, we have $\mathcal{T}_X\cong p^*\mathcal{T}_{\mathbb{P}^1}\oplus q^*\mathcal{T}_Y.$
Hence, $\wedge^3\mathcal{T}_X\cong p^*\mathcal{T}_{\mathbb{P}^1}\otimes q^*(\wedge^2\mathcal{T}_Y)\oplus q^*(\wedge^3\mathcal{T}_Y)$ is nef because we can easily check that $\wedge^3\mathcal{T}_Y$ is nef. 
\end{example}

\begin{example}\label{third5}
Let $X=Bl_{l}(Q^n)$ be the blowing-up of $Q^n$ along a line $l$ where $n\geqslant4$.
At first, we show that $\wedge^{n-1}\mathcal{T}_X$ is nef.
Let $Y=Bl_{l}(\mathbb{P}^{n+1})$ be the blowing-up of $\mathbb{P}^{n+1}$ containing $Q^n$ along a line $l$.
Then, $X$ is a closed subvariety of $Y$ of codimension one.
$Y$ has a $\mathbb{P}^2$-bundle structure $Y\cong\mathbb{P}_{\mathbb{P}^{n-1}}(\mathcal{O}^2_{\mathbb{P}^{n-1}}\oplus\mathcal{O}_{\mathbb{P}^{n-1}}(1))$. 
Put $\xi$ the tautological divisor and $H$ a divisor associated with the line bundle $\pi^*\mathcal{O}_{\mathbb{P}^{n-1}}(1)$ where 
$\pi : \mathbb{P}_{\mathbb{P}^{n-1}}(\mathcal{O}^2_{\mathbb{P}^{n-1}}\oplus\mathcal{O}_{\mathbb{P}^{n-1}}(1))\rightarrow \mathbb{P}^{n-1}$ is the natural projection.
We note that fibers of $\pi$ correspond to planes in $\mathbb{P}^{n+1}$ containing $l$.
From this fact, we can see that $\pi^{-1}(z)\cap X\cong\mathbb{P}^1$ for a general $z\in\mathbb{P}^{n-1}$. 

We show that $X\sim \xi+H$ in $Y$.
We set that $X\sim a\xi+bH$ for some integers $a$ and $b$.
Let $C$ be the strict transform of a line in $P^{n+1}$ not contained in $Q^n$ such that it does not intersect with $l$.
Then, we have that $X.C=2$, $\xi.C=1$ and $H.C=1$.
Let $C'$ be the strict transform of a line in $P^{n+1}$ not contained in $Q^n$ such that it meets $l$ at one point.
Then, we have that $X.C'=2-1=1$, $\xi.C'=1$ and $H.C'=1-1=0$.
Hence, we obtain that $a=b=1$. 
From this, we know that 
\[\mathcal{N}_{X/Y}= \xi+H|_X.\] 

By the exact sequence $0\rightarrow \mathcal{T}_X\rightarrow \mathcal{T}_Y|_X\rightarrow \mathcal{N}_{X/Y}\rightarrow 0$, we get the surjection of vector bundles on $X$,
$\wedge^{n}\mathcal{T}_Y|_X\otimes\mathcal{N}_{X/Y}^{\vee}\rightarrow \wedge^{n-1}\mathcal{T}_{X}\rightarrow 0$.
We confirm that $\wedge^{n}\mathcal{T}_Y(-\xi-H)|_X\cong\wedge^{n}\mathcal{T}_Y|_X\otimes\mathcal{N}_{X/Y}^{\vee}$ is nef.
From the exact sequence
\[0\rightarrow \mathcal{T}_{\pi}\rightarrow \mathcal{T}_Y\rightarrow \pi^*\mathcal{T}_{\mathbb{P}^{n-1}}\rightarrow 0, \]
we have the following exact sequence  
\[0\rightarrow\wedge^2\mathcal{T}_{\pi}(-\xi-H)\otimes\pi^*(\wedge^{n-2}\mathcal{T}_{\mathbb{P}^{n-1}})\rightarrow \wedge^n\mathcal{T}_Y(-\xi-H)\rightarrow \mathcal{T}_{\pi}(-\xi-H)\otimes\pi^*(\wedge^{n-1}\mathcal{T}_{\mathbb{P}^{n-1}})\rightarrow 0. \]
 The subbundle $\wedge^2\mathcal{T}_{\pi}(-\xi-H)\otimes\pi^*(\wedge^{n-2}\mathcal{T}_{\mathbb{P}^{n-1}})\cong\mathcal{O}_Y(2\xi)\otimes\pi^*(\wedge^{n-2}\mathcal{T}_{\mathbb{P}^{n-1}}(-2))$ is nef.
To show the nefness of the quotient bundle, we use the exact sequence,
\[0\rightarrow \mathcal{O}_{Y}\rightarrow \mathcal{O}_Y(\xi)\otimes\pi^*(\mathcal{O}^2_{\mathbb{P}^{n-1}}\oplus\mathcal{O}_{\mathbb{P}^{n-1}}(-1))\rightarrow \mathcal{T}_{\pi}\rightarrow 0. \]
We can check that the quotient bundle is also nef by the surjection from the nef vector bundle, 
\[\pi^*(\wedge^{n-1}\mathcal{T}_{\mathbb{P}^{n-1}}(-1))^{\oplus 2}\oplus\pi^*(\wedge^{n-1}\mathcal{T}_{\mathbb{P}^{n-1}}(-2))\rightarrow \mathcal{T}_{\pi}(-\xi-H)\otimes\pi^*(\wedge^{n-1}\mathcal{T}_{\mathbb{P}^{n-1}})\rightarrow 0.\] 
Therefore, $\wedge^{n}\mathcal{T}_Y(-\xi-H)$ is nef.  
Hence, we have that $\wedge^{n-1}\mathcal{T}_X$ is nef. 

Next, we confirm that $\wedge^{i}\mathcal{T}_X$ is not nef for $i\leqslant n-2$.
We suppose that $\wedge^{i}\mathcal{T}_X$ is nef for $i\leqslant n-2$. 
We have $X \sim \xi+H$ in $Y$.
 Sections of $\mathcal{O}_Y(\xi+H)$ correspond to sections of $\mathcal{O}^2_{\mathbb{P}^{n-1}}(1)\oplus\mathcal{O}_{\mathbb{P}^{n-1}}(2)$. 
 Since $c_3(\mathcal{O}^2_{\mathbb{P}^{n-1}}(1)\oplus\mathcal{O}_{\mathbb{P}^{n-1}}(2))=2\not=0$, we know that any section of $\mathcal{O}_Y(\xi+H)$ is zero on some fiber of $\pi$. 
 Hence, the elementary contraction $\pi|_X : X \rightarrow \mathbb{P}^{n-1}$ has a two dimensional fiber and corresponding extremal rational curve $C$ satisfies $-K_X.C=2$.
 This contradicts to Lemma \ref{p-bdle}. 
\end{example}

\begin{lemma}\label{rest}
Let $X$ be an $n$-dimensional smooth projective manifold with nef vector bundle $\displaystyle \wedge ^{i} \mathcal{T}_X$ for $i\leqslant n-1$.
Put $C$ be a rational curve in $X$ and $\nu:\mathbb{P}^1\rightarrow C$ the normalization.
Then $-K_X.C\geqslant 2$.  
We further assume that $i\leqslant n-2$ and $-K_X.C=2$, then $\nu^*\mathcal{T}_X\cong\mathcal{O}_{\mathbb{P}^1}(2)\oplus\mathcal{O}^{n-1}_{\mathbb{P}^1}$.
\end{lemma}
\begin{proof}
Set $\nu^*\mathcal{T}_X\cong\mathcal{O}_{\mathbb{P}^1}(a_1)\oplus\mathcal{O}_{\mathbb{P}^1}(a_2)\oplus\cdots\oplus\mathcal{O}_{\mathbb{P}^1}(a_n)$ where $a_1\geqslant a_2\geqslant\cdots\geqslant a_n$.
We know that $a_1\geqslant 2$ and $a_{n-i+1}+\cdots+a_n\geqslant 0$.
Therefore we have $-K_X.C\geqslant 2$.  
We assume that $i\leqslant n-2$ and $-K_X.C=2$.
If $a_n<0$, then we get $-K_X.C=a_1+a_2+\cdots+a_n\geqslant n-i+1\geqslant 3$.
Therefore, we have the assertion.
\end{proof}

\begin{lemma}\label{p-bdle}
Let $X$ be an $n$-dimensional smooth projective manifold with nef vector bundle $\wedge^{i}\mathcal{T}_X$ for $i\leqslant n-2$.
Let $C$ be an rational curve in $X$ such that $-K_X.C=2$.
Then $C$ is in some extremal ray and corresponding elementary contraction $\varphi:X\rightarrow Y$ is a $\mathbb{P}^1$-bundle over a smooth projective manifold $Y$.  
\end{lemma}

\begin{proof}
The proof is based on the argument in the proof of Theorem 8 in \cite{mori}. 
Let $\nu:\mathbb{P}^1\rightarrow C\subset X$ be the normalization.
Let $H$ be an irreducible component of the Hom scheme $Hom(\mathbb{P}^1,X)$ containing the morphism $\nu$.
In this case $h^*\mathcal{T}_X\cong\mathcal{O}_{\mathbb{P}^1}(2)\oplus\mathcal{O}^{n-1}_{\mathbb{P}^1}$ for 
all $h\in H$ by Lemma \ref{rest}. 
Then we see that $\dim H=n+2$ and $H$ is smooth by Proposition 3 in \cite{mori}.
Let $G$ be the Aut($\mathbb{P}^1$).
We can construct an $(n-1)$-dimensional smooth projective manifold $Y($resp. an $n$-dimensional smooth projective manifold $Z )$
which is the geometric quotient of $H($resp. $H\times\mathbb{P}^1)$ by $G$ from the argument in the proof of Theorem 8 in \cite{mori}.
Moreover, we have a morphisms $q:Z\rightarrow Y$ which is $\mathbb{P}^1$-bundle in the \'etale topology 
and the evaluation morphism $p:Z\rightarrow X$.
We can show that $p$ is a smooth surjective morphism from $n$-dimensional smooth manifold to $n$-dimensional Fano manifold.
 Since any Fano manifold is simply connected, $p$ is isomorphism and we finish the proof.  
\end{proof}

\begin{example}\label{third6}
Let $X=Bl_{C}(Q^n)$ be the blowing-up of $Q^n\subseteq\mathbb{P}^{n+1}$ along a conic $C$ not on a plain contained in $Q^n$ where $n\geqslant 4$.
At first, we show that $\wedge^{n-1}\mathcal{T}_X$ is nef.
Take a plain $S\cong\mathbb{P}^2$ which contains $C$. 
Let $Y=Bl_{S}(\mathbb{P}^{n+1})$ be the blowing-up of $\mathbb{P}^{n+1}$ containing $Q^n$ along a plain $S$.
Then, $X$ is a closed subvariety of $Y$ of codimension one.
$Y$ has a $\mathbb{P}^3$-bundle structure $Y\cong\mathbb{P}_{\mathbb{P}^{n-2}}(\mathcal{O}^3_{\mathbb{P}^{n-2}}\oplus\mathcal{O}_{\mathbb{P}^{n-2}}(1))$. 
Put $\xi$ the tautological divisor and $H$ a divisor associated with the line bundle $\pi^*\mathcal{O}_{\mathbb{P}^{n-2}}(1)$ where 
$\pi : \mathbb{P}_{\mathbb{P}^{n-2}}(\mathcal{O}^3_{\mathbb{P}^{n-2}}\oplus\mathcal{O}_{\mathbb{P}^{n-2}}(1))\rightarrow \mathbb{P}^{n-2}$ is the natural projection.
We note that fibers of $\pi$ correspond to 3-planes in $\mathbb{P}^{n+1}$ containing $S$.
From this fact, we can see that $\pi^{-1}(z)\cap X$ is an $($possibly singular$)$ irreducible and reduced quadric surface for any $z\in\mathbb{P}^{n-2}$. 

We show that $X\sim 2\xi$ in $Y$.
We set that $X\sim a\xi+bH$ for some integers $a$ and $b$.
Let $C'$ be the strict transform of a line in $P^{n+1}$ not contained in $Q^n$ such that it does not intersect with $C$.
Then, we have that $X.C'=2$, $\xi.C=1$ and $H.C=1$.
Let $C''$ be the strict transform of a line in $P^{n+1}$ not contained in $Q^n$ such that it meets $C$ at one point.
Then, we have that $X.C''=2$, $\xi.C''=1$ and $H.C''=0$.
Hence, we obtain that $a=2$ and $b=0$. 
From this, we know that 
\[\mathcal{N}_{X/Y}= 2\xi|_X.\] 

By the exact sequence $0\rightarrow \mathcal{T}_X\rightarrow \mathcal{T}_Y|_X\rightarrow \mathcal{N}_{X/Y}\rightarrow 0$, we get the surjection of vector bundles on $X$,
$\wedge^{n}\mathcal{T}_Y|_X\otimes\mathcal{N}_{X/Y}^{\vee}\rightarrow \wedge^{n-1}\mathcal{T}_{X}\rightarrow 0$.
We confirm that $\wedge^{n}\mathcal{T}_Y(-2\xi)|_X\cong\wedge^{n}\mathcal{T}_Y|_X\otimes\mathcal{N}_{X/Y}^{\vee}$ is nef.
From the exact sequence
\[0\rightarrow \mathcal{T}_{\pi}\rightarrow \mathcal{T}_Y\rightarrow \pi^*\mathcal{T}_{\mathbb{P}^{n-2}}\rightarrow 0, \]
we have the following exact sequence  
\[0\rightarrow(\wedge^3\mathcal{T}_{\pi})(-2\xi)\otimes\pi^*(\wedge^{n-3}\mathcal{T}_{\mathbb{P}^{n-2}})\rightarrow \wedge^n\mathcal{T}_Y(-2\xi)\rightarrow (\wedge^2\mathcal{T}_{\pi})(-2\xi)\otimes\pi^*(\wedge^{n-2}\mathcal{T}_{\mathbb{P}^{n-2}})\rightarrow 0. \]
 The subbundle $\wedge^3\mathcal{T}_{\pi}(-2\xi)\otimes\pi^*(\wedge^{n-3}\mathcal{T}_{\mathbb{P}^{n-2}})\cong\mathcal{O}_Y(2\xi)\otimes\pi^*(\wedge^{n-3}\mathcal{T}_{\mathbb{P}^{n-1}}(-1))$ is nef.
To show the nefness of the quotient bundle, we use the exact sequence,
\[0\rightarrow \mathcal{O}_{Y}\rightarrow \mathcal{O}_Y(\xi)\otimes\pi^*(\mathcal{O}^3_{\mathbb{P}^{n-2}}\oplus\mathcal{O}_{\mathbb{P}^{n-2}}(-1))\rightarrow \mathcal{T}_{\pi}\rightarrow 0. \]
We can check that the quotient bundle is also nef by the surjection from the nef vector bundle, 
\[\pi^*(\wedge^{n-2}\mathcal{T}_{\mathbb{P}^{n-2}})^{\oplus 3}\oplus\pi^*(\wedge^{n-2}\mathcal{T}_{\mathbb{P}^{n-2}}(-1))^{\oplus 3}\rightarrow (\wedge^2\mathcal{T}_{\pi})(-2\xi)\otimes\pi^*(\wedge^{n-2}\mathcal{T}_{\mathbb{P}^{n-2}})\rightarrow 0.\] 
Therefore, $\wedge^{n}\mathcal{T}_Y(-2\xi)$ is nef.  
Hence, we have that $\wedge^{n-1}\mathcal{T}_X$ is nef. 

Next, we confirm that $\wedge^{i}\mathcal{T}_X$ is not nef for $i\leqslant n-2$.
We suppose that $\wedge^{i}\mathcal{T}_X$ is nef for $i\leqslant n-2$. 
We have the elementary contraction $\pi|_X : X \rightarrow \mathbb{P}^{n-2}$ such that it has a two dimensional fiber and corresponding extremal rational curve $C'$ satisfies $-K_X.C'=2$.
 This contradicts to Lemma \ref{p-bdle}. 
\end{example}

\section{Proof of Main Theorem 1}

In this section, we give a proof of Main Theorem 1. 
\begin{theorem}
Let $X$ be an $n$-dimensional Fano manifold with at least one elementary contraction of birational type such that the second exterior power of tangent bundle $\wedge^2\mathcal{T}_X$ is nef where $n\geqslant 3$.
Then $X$ is isomorphic to the blowing-up of the $n$-dimensional projective space $\mathbb{P}^n$ at a point. 
\end{theorem}  

If part is proved in Example \ref{bl_pt}.
Therefore, we prove only if part.

Let $X$ be an $n$-dimensional Fano manifold with an elementary contraction of birational type $\tau : X \rightarrow Y$ such that the second exterior power of tangent bundle $\wedge^2\mathcal{T}_X$ is nef where $n\geqslant 3$.

\begin{claim}\label{rest2}
There is an extremal rational curve contracted by $\tau$ such that $n+1\geqslant -K_X.C\geqslant n-1$.  
\end{claim}
\begin{proof}
By Bend and Break Lemma we have a rational curve $C$ contracted by $\tau$ such that $n+1\geqslant -K_X.C$ and $\mathcal{T}_X|_C$ is not nef since $\tau$ is of birational type. 
Put $\nu:\mathbb{P}^1\rightarrow C$ the normalization.
Set $\nu^*\mathcal{T}_X\cong\mathcal{O}_{\mathbb{P}^1}(a_1)\oplus\mathcal{O}_{\mathbb{P}^1}(a_2)\oplus\cdots\oplus\mathcal{O}_{\mathbb{P}^1}(a_n)$ where $a_1\geqslant a_2\geqslant\cdots\geqslant a_n$.
We know that $a_1\geqslant 2$, $a_{n-1}+a_n\geqslant 0$ and $a_n<0$.
In particular, $a_{n-1}\geqslant 1$.
Therefore we get $-K_X.C=a_1+a_2+\cdots+a_n\geqslant n-1$.
\end{proof} 
Set $E$ an irreducible component of the exceptional locus of $\tau$.
Then we know that $(\dim E, \dim\tau(E))=(n-1,0)$ by the following proposition proved in \cite{wisniewski}.
Im particular, $\tau$ is a divisorial contraction. 
\begin{proposition}\label{wisniewski}
Let X be a smooth projective manifold of dimension $n$ and $R$ an extremal ray on $X$ which corresponds to the contraction $\tau:X\rightarrow Y$.
Set $l(R):=\min \{-K_X.C|[C]\in R\}$.
Let $E$ be any irreducible component of the exceptional locus of $\tau$ and $F$ any irreducible component of any nontrivial fiber of $\tau$.
Then
\[\dim E+\dim F\geqslant n+l(R)-1.\]  
\end{proposition}

Let $E$ be the exceptional divisor of $\tau$.
Then, we can find a rational curve $C$ of minimal degree in some extremal ray $R$ such that $E.C>0$ by Cone theorem since $X$ is Fano and $E$ is effective and nontrivial.  
Let $\phi : X\rightarrow Z$ be the elementary contraction associated with $R$.
Then, any fiber of $\phi$ is of dimension $\leqslant 1$. 
Since $-K_X.C'>1$ for any rational curve $C'$ on $X$ we have that $\phi$ is $\mathbb{P}^1$-bundle by Lemma 2.1 of \cite{fujita}.
Moreover we know that $E$ is the section of $\phi$ by Claim 2.3 in \cite{fujita}.
Hence $E\cong Z$ is smooth and there is a rank 2 vector bundle $\mathcal{E}$ on $Z$ such that $X\cong\mathbb{P}_Z(\mathcal{E})$.
Since $n\geqslant-K_E.C=-K_X|_E.C-\mathcal{N}_{E/X}.C\geqslant(n-1)+1=n$, we have that $C$ is a line in $E$ and $-K_E=nL$ for some ample line bundle $L\in Pic(E)$ by the following theorem. 
\begin{theorem}[\cite{ando}, Theorem 2.1]\label{ando1}
Let $X$ be a smooth projective manifold of dimension n. Let $f:X\rightarrow Y$ be an extremal contraction and $E$ the exceptional locus of $f$. Assume that $\mathrm{dim}E=n-1$.
Let F be a general fiber of $f_E:E\rightarrow f(E)$. Then there exists a Cartier divisor $L$ on $X$ such that 
\begin{enumerate}
\item Im$($Pic$X\rightarrow$Pic$F)=\mathbb{Z}[L|_F]$ and $L|_F$ is ample on $F$ $;$
\item $\mathcal{O}_F(-K_X)\cong\mathcal{O}_F(pL)$ and $\mathcal{O}_F(-E)\cong\mathcal{O}_F(qL)$ for some $p,q\in\mathbb{N}.$ 
\end{enumerate}
\end{theorem}

Therefore, we have $E\cong\mathbb{P}^{n-1}$ by Kobayashi-Ochiai's characterization \cite{kobayashi} and $\mathcal{N}_{E/X}\cong\mathcal{O}_{\mathbb{P}^{n-1}}(-1)$.  
Taking the push-forward of the exact sequence
\[0\rightarrow \mathcal{O}_X\rightarrow\mathcal{O}_X(E)\rightarrow\mathcal{N}_{E/X}\rightarrow 0\]
by $\phi$, we have that $X\cong\mathbb{P}_{\mathbb{P}^{n-1}}(\mathcal{O}_{\mathbb{P}^{n-1}}\oplus\mathcal{O}_{\mathbb{P}^{n-1}}(-1))$ which is the blowing up of $\mathbb{P}^n$ at a point.
Hence we complete the proof.

\section{Proof of Main Theorem 2}

In this section, we give a proof of Main Theorem 2. 
\begin{theorem}
Let $X$ be an $n$-dimensional Fano manifold with at least one elementary contraction of birational type such that the third exterior power of tangent bundle $\wedge^3\mathcal{T}_X$ is nef where $n\geqslant 4$.
Assume that $\wedge^2\mathcal{T}_X$ is not nef.
Then X is isomorphic to one of the following,
\begin{enumerate}
\item $\mathbb{P}_{\mathbb{P}^{n-1}}(\mathcal{O}_{\mathbb{P}^{n-1}}\oplus \mathcal{O}_{\mathbb{P}^{n-1}}(-2))$,
\item $\mathbb{P}_{Q^{n-1}}(\mathcal{O}_{Q^{n-1}}\oplus \mathcal{O}_{Q^{n-1}}(-1))$ where $Q^{n-1}$ is a $(n-1)$-dimensional smooth quadric hypersurface in $\mathbb{P}^n$, 
\item blowing-up of $\mathbb{P}^n$ along a line, 
\item product of $\mathbb{P}^1$ and the blowing-up of $\mathbb{P}^{n-1}$ at a point, 
\item blowing-up of $Q^4$ along a line, 
\item blowing-up of $Q^4$ along a conic not on a plain contained in $Q^4$, 
\end{enumerate}
\end{theorem}
If part is proved in Example \ref{third1}, \ref{third2}, \ref{third3}, \ref{third4}, \ref{third5} and \ref{third6}.
Therefore, we prove only if part.

Let $X$ be an $n$-dimensional Fano manifold with an elementary contraction of birational type $\tau : X \rightarrow Y$ such that the third exterior power of tangent bundle $\wedge^3\mathcal{T}_X$ is nef where $n\geqslant 4$.
Furthermore we suppose that $\wedge^2\mathcal{T}_X$ is not nef.
Put $C$ a rational curve of minimal degree in an extremal ray which corresponds to the given birational contraction $\tau$. 

\begin{claim}\label{rest2}
There is an extremal rational curve contracted by $\tau$ such that $n+1\geqslant -K_X.C\geqslant n-2$.  
\end{claim}
\begin{proof}
By Bend and Break Lemma we have an extremal rational curve $C$ contracted by $\tau$ such that $n+1\geqslant -K_X.C$ and $\mathcal{T}_X|_C$ is not nef since $\tau$ is of birational type. 
Put $\nu:\mathbb{P}^1\rightarrow C$ the normalization.
Set $\nu^*\mathcal{T}_X\cong\mathcal{O}_{\mathbb{P}^1}(a_1)\oplus\mathcal{O}_{\mathbb{P}^1}(a_2)\oplus\cdots\oplus\mathcal{O}_{\mathbb{P}^1}(a_n)$ where $a_1\geqslant a_2\geqslant\cdots\geqslant a_n$.
We know that $a_1\geqslant 2$, $a_{n-2}+a_{n-1}+a_n\geqslant 0$ and $a_n<0$.
In particular, $a_{n-2}\geqslant 1$.
Therefore we get $-K_X.C=a_1+a_2+\cdots+a_n\geqslant n-2$.
\end{proof}
Set $E$ an irreducible component of the exceptional locus of $\tau$.
Then we know that $(\dim E, \dim\tau(E))=(n-1,0)$ or $(n-1,1)$ by Proposition \ref{wisniewski}.
In particular, $\tau$ is a divisorial contraction. 

Let $E$ be the exceptional divisor of $\tau$.
Then, we can find a rational curve $C$ of minimal degree in some extremal ray $R$ such that $E.C>0$ by Cone theorem since $X$ is Fano and $E$ is effective and nontrivial.  
Let $\phi : X\rightarrow Z$ be the elementary contraction associated with $R$.
Then, any fiber of $\phi$ is of dimension $\leqslant 2$. 

We consider the case where $(\dim E, \dim\tau(E))=(n-1,0)$.
Then, any fiber of $\phi$ is of dimension $\leqslant 1$.
Therefore we have $\phi$ is $\mathbb{P}^1$-bundle, $Z$ is a Fano manifold and $E$ is the section of $\phi$ by the same argument in the proof of Theorem $1.1$.
Hence $E\cong Z$ is smooth and there is a rank 2 vector bundle $\mathcal{E}$ on $Z$ such that $X\cong\mathbb{P}_Z(\mathcal{E})$.
Since $n\geqslant-K_E.C=-K_X|_E.C-\mathcal{N}_{E/X}.C\geqslant(n-2)+1=n-1$, we have $(-K_E.C,-K_X|_E.C,-\mathcal{N}_{E/X}.C)=(n,n-1,1)$, $(n-1,n-2,1)$ or $(n, n-2, 2)$.

If $(-K_E.C,-K_X|_E.C,-\mathcal{N}_{E/X}.C)=(n,n-1,1)$, then we have $X\cong Bl_{point}\mathbb{P}^n$ by the proof of Main theorem 1. 
In this case, we know $\wedge^2\mathcal{T}_X$ is nef by Example \ref{bl_pt}.
This is a contradiction. 

If $(-K_E.C,-K_X|_E.C,-\mathcal{N}_{E/X}.C)=(n-1,n-2,1)$, we have that $C$ is a line in $E$ and $-K_E=(n-1)L$ for some ample line bundle $L\in Pic(E)$ by Theorem \ref{ando1}. 
Therefore, we have $E\cong Q^{n-1}$ by Kobayashi-Ochiai's characterization \cite{kobayashi} and $\mathcal{N}_{E/X}\cong\mathcal{O}_{Q^{n-1}}(-1)$.  
Taking the push-forward of the exact sequence
\[0\rightarrow \mathcal{O}_X\rightarrow\mathcal{O}_X(E)\rightarrow\mathcal{N}_{E/X}\rightarrow 0\]
by $\phi$, we have that $X\cong\mathbb{P}_{Q^{n-1}}(\mathcal{O}_{Q^{n-1}}\oplus\mathcal{O}_{Q^{n-1}}(-1))$.

We consider the case where $(-K_E.C,-K_X|_E.C,-\mathcal{N}_{E/X}.C)=(n,n-2,2)$.
We have $-K_E.C'\geqslant n=\dim E+1$ for any rational curve in $E$ because $C'$ is numerically equivalent to $aC$ for some positive integer $a$.
From the characterization of projective spaces due to Cho--Miyaoka--Shepherd-Barron\cite{cho}, we know $E\cong\mathbb{P}^{n-1}$ and $\mathcal{N}_{E/X}\cong\mathcal{O}_{\mathbb{P}^{n-1}}(-2)$.
Taking the push-forward of the exact sequence
\[0\rightarrow \mathcal{O}_X\rightarrow\mathcal{O}_X(E)\rightarrow\mathcal{N}_{E/X}\rightarrow 0\]
by $\phi$, we have that $X\cong\mathbb{P}_{\mathbb{P}^{n-1}}(\mathcal{O}_{\mathbb{P}^{n-1}}\oplus\mathcal{O}_{\mathbb{P}^{n-1}}(-2))$.

We have the assertion if $(\dim E, \dim\tau(E))=(n-1,0)$. 

Secondary, we consider the case where $(\dim E, \dim\tau(E))=(n-1,1)$.  
In this case, $Y$ and $W:=\tau(E)$ are smooth and $\tau$ is the blowing-up of $Y$ along a smooth curve $W$ by Theorem 5.1 in \cite{andreatta} and Claim \ref{rest2}.  
By Lemma \ref{rest}, we know that  the pseudo-index $i(X):=\min\{-K_X.C|$ $C$ is a rational curve in X$\}\geqslant 2$.

From the classification of Tsukioka \cite{tsukioka}, we have that $X$ is isomorphic to one of the following,
\begin{enumerate} 
\item blowing-up of $\mathbb{P}^n$ along a line, 
\item product of $\mathbb{P}^1$ and the blowing-up of $\mathbb{P}^{n-1}$ at a point, 
\item blowing-up of $Q^n$ along a line, 
\item blowing-up of $Q^n$ along a conic not on a plain contained in $Q^n$,
\end{enumerate}
Note that we use the fact that if $X$ is a blowing-up of a smooth $n$-dimensional manifold along a smooth subvariety of codimension 2, then $i(X)=1$.
By Example \ref{third3}, \ref{third4}, \ref{third5} and \ref{third6}, we have the assertion.


\end{document}